\documentclass{amsart}
\usepackage{amsthm}
\usepackage[utf8]{inputenc}
\usepackage{appendix}
\usepackage{cite}
\usepackage{amssymb,amsmath,amsthm}

\newcommand{\amsprimary}[1]{{\footnotesize\noindent AMS 2020 \textit{Mathematics subject
classification:} Primary #1\vspace{1pc}}}
\newcommand{\keywordsnames}[1]{{\footnotesize\noindent\textit{Key words:} #1\vspace{1pc}}}

\newtheorem{theorem}{Theorem}
\newtheorem{teo}{Theorem}

\newtheorem{lemma}[teo]{Lemma}

\theoremstyle{definition}

\theoremstyle{remark}

\title[sublinear elliptic systems]{On Sublinear elliptic systems on bounded and thin unbounded domains}
\author{Jean C. Cortissoz }
\address{Universidad de los Andes, Bogot\'a DC, COLOMBIA}
\email{jcortiss@uniandes.edu.co}
\date{}

\begin{document}

\begin{abstract}

In this paper, we demonstrate the existence of positive solutions for certain weakly coupled elliptic systems of sublinear growth under homogeneous Dirichlet boundary conditions. Our findings generalize existing results related to sublinear systems involving two weakly coupled equations, such as the Lane-Emden systems. Moreover, our results apply to both bounded domains and thin unbounded domains, defined as unbounded regions contained between two parallel hyperplanes, and accommodate some discontinuous nonlinearities.

\end{abstract}

\maketitle
 
{\keywordsnames {.} Sublinear elliptic systems, Dirichlet boundary conditions}

{\amsprimary {35J57, 35J60}}

\section{Introduction}

In this paper, we consider a system that generalizes 
the well-studied sublinear Lane-Emden system, which, in the context described below, corresponds
to $n=1$, $\lambda_l\equiv 1$ and $f_l\left(u_0,\dots, u_n\right)=u_{l+1}^{p_{l+1}}$
for $l=0,\dots, n-1$ and $f_n\left(u_0,\dots, u_n\right)=u_0^{p_0}$.
\begin{equation}
\label{eq:gen_lane_emden}
\left\{
    \begin{array}{l}
    -\Delta u_{l}=\lambda_{l}\left(x\right)f_l\left(u_0,u_1,\dots, u_n\right), \quad l=0,1,2,\dots, n, 
    \quad \mbox{in }\quad \Omega\\
    u_l=0, \quad l=0,1,2,\dots,n, \quad \mbox{on} \quad \partial \Omega,
    \end{array}
    \right.
\end{equation}
where $\Omega$ is either a bounded or possibly unbounded thin subdomain (a concept to be defined below) 
of $\mathbb{R}^n$ with smooth boundary.

\medskip
The Lane-Emden equation, which originated in the study of self-gravitating
fluids in astrophysics (for instance, stars),
and the Lane-Emden system (that is, in the case $n=1$) as well as sublinear
elliptic equations and systems have been extensively
studied in the literature (see, for example, \cite{Barile, Brasco, Figueiredo, Kajikija, Ramos}). 
In most cases, a variational approach has been used to study the
existence of nontrivial solutions. In this paper, we 
not only generalize the form of the system, including the
number of possible equations, but our method also provides a way 
to produce explicit estimates on
the solutions that are constructed, as it is based on a monotone iteration method coupled with
a maximum principle. 

\medskip
Our results relate to those of \cite{Ma} and \cite{Cid}, offering new insights: both \cite{Ma} and \cite{Cid} require the functions $f_l$ to be continuous (with \cite{Ma} demanding a specific form for the functions $f_l$) and that the domain
$\Omega$ be bounded. In contrast, our findings allow for weak discontinuities in bounded domains and also address the case of unbounded thin domains, which has not been as thoroughly documented in the literature as that of bounded domains.

\medskip
We focus on a specific class of unbounded domains, which is a notable feature of our method. For a system of two equations, the study of unbounded cylinders—which are a particular case of a thin unbounded domain—while assuming certain symmetries in the nonlinearities is explored in \cite{Barile2}. The techniques utilized in this paper were first developed in \cite{Cortissoz}, and we extend their scope to include more general nonlinearities.

\medskip

In order to state our
main results, we begin by imposing the following conditions on the $f_l$'s in (\ref{eq:gen_lane_emden}).

\medskip
\begin{itemize}

    \item[(a)] The $f_l$'s are measurable functions such that
    \[
    \left|f_l\left(z_0,z_1,\dots,z_{j-1},z_j, z_{j+1},\dots, z_n\right)\right|
    \leq C_{l,j}\left|z_j\right|^{p_{l,j}}+A,
    \]
     with $0<p_{l,j}<1$, $A\geq 0$.

    \item[(b)]
    There is an $\epsilon_0>0$ such that the following holds:
    There are nonnegative constants $A_{l,j}$ such that for each fixed 
    $l$ at least one is different from 0, and $0<\alpha_{l,j}<1$ for which
    \[
    f_l\left(z_0,z_1,\dots, z_n\right)\geq \sum_{j=0}^{n}A_{l,j}\left|z_j\right|^{\alpha_{l,j}},
    \]
    whenever $\left|z_j\right|<\epsilon_0$, $j=0,\dots,n$.

    \item[(c)] The $f_l$'s are increasing in the following sense. If $z_j\geq y_j$, $j=0,1,\dots, n$, then
    \[
    f\left(y_0,y_1,\dots,y_n\dots\right)\leq f\left(z_0,z_1,\dots, z_n\right).
    \] 
    Systems that satisfy this condition are generally called cooperative.

    \item[(d)] The $f_l$'s satisfy the following form of continuity (and thus
    allows for certain discontinuous nonlinearities): If $\left(z_{j,k}\right)_{k=1, 2, \dots}$
    $j=0,1,\dots,n$ are nondecreasing sequences of real numbers such that $z_{j,k}\uparrow z_j$ then
    \[
    \lim_{k\rightarrow\infty} f_l\left(z_{0,k},\dots,z_{n,k}\right)=f\left(z_0,\dots,z_n\right).
    \]
\end{itemize}
We shall say that an unbounded domain $\Omega$ is a \emph{thin domain} if it is contained in the region
bounded between 
two parallel hyperplanes (thus, all bounded domains
are thin). We define the \emph{slab diameter} of an unbounded thin domain $\Omega$ as the infimum
of the
distances of parallel hyperplanes such that $\Omega$ is contained in the region between them.
The definition of slab diameter can also be applied to bounded domains. 

\medskip
Under the previous assumptions on the $f_l$'s, we shall prove the following theorem.
\begin{theorem}
\label{thm:main}
    Let $\Omega$ be a bounded subdomain of $\mathbb{R}^n$. Assume that the $f_j$'s satisfy assumptions (a)-(d), that $\lambda_j\in L^{\infty}\left(\Omega\right)$ with $\lambda_j\geq 0$, and that there is a point $q\in \Omega$
    and $\rho>0$, $l=0,1,\dots,n$, such that $\lambda_l\left(x\right)\geq \Lambda>0$ for all $x\in B_{2\rho}\left(q\right)$.
    Then (\ref{eq:gen_lane_emden}) has a non-trivial non-negative solution.
\end{theorem}

Under stronger assumptions on $f$, we can show the following existence theorem on thin unbounded
domains.

\begin{theorem}
\label{thm:main2}
    Let $\Omega$ be an unbounded thin subdomain of $\mathbb{R}^n$. Assume that the
    $f_l$'s are H\"older continuous and that they satisfy assumptions (a)-(c). Furthermore,
    assume that each
    $\lambda_l\in L^{\infty}\left(\Omega\right)$ is nonnegative and H\"older continuous,
    and that there is a point $q\in \Omega$
    and $\rho>0$, $l=0,1,\dots,n$, such that 
    $\lambda_l\left(x\right)\geq \Lambda>0$ for all $x\in B_{2\rho}\left(q\right)$.
    Then (\ref{eq:gen_lane_emden}) has a non-trivial non-negative solution.
\end{theorem}

If the $\lambda_l$'s and the $f_l$'s are locally H\"older continuous, that is, if for some $0<\alpha<1$,
for each $l$, $\lambda_l\in C_{loc}^{\alpha}\left(\overline{\Omega}\right)$ and
$f_l\in C_{loc}^{\alpha}\left(\mathbb{R}^n\right)$
by a solution 
to (\ref{eq:gen_lane_emden}) in both cases, when $\Omega$ is 
a bounded or an unbounded domain, we mean a classical solution. When the 
$\lambda_l$'s are in $L^{\infty}\left(\Omega\right)$, in the case of $\Omega$ being a 
\emph{bounded domain}, by a solution we mean the
following:

\medskip
We say that $\left(u_0,u_1,\dots, u_n\right)\in W_0^{1,2}(\Omega)^{n+1}\cap L^{\infty}(\Omega)^{n+1}$
is a weak solution to (\ref{eq:gen_lane_emden}) if
for every $\left(\varphi_0,\varphi_1,\dots, \varphi_n\right)\in C_0^{\infty}\left(\Omega\right)^{n+1}$
it holds that
\[
\int_{\Omega}\nabla u_j \nabla \varphi_j= \int_{\Omega}\lambda_{j}\left(x\right)
f_j\left(u_0, u_1,\dots, u_n\right)\varphi_{j}\,dV, \quad j=0,1,\dots,n.
\]
The proofs of our results are based on an iteration method, and the main technical difficulty
lies in constructing a subsolution to (\ref{eq:gen_lane_emden}), that is, a function
$h$ such that 
\begin{equation*}
\left\{
    \begin{array}{l}
    -\Delta h\leq\lambda_{l}\left(x\right)f_l\left(h,h,\dots, h\right), \quad l=0,1,2,\dots, n, 
    \quad \mbox{in }\quad \Omega,\\
    h=0, \quad l=0,1,2,\dots,n, \quad \mbox{on} \quad \partial \Omega,
    \end{array}
    \right.
\end{equation*}
to start the iteration: here condition (b) is essential. We refer to the function
$h$ as a \emph{subsolution to (\ref{eq:gen_lane_emden})}. Condition (a) is used to prove 
that the iterates are bounded, and, as shown in the last section of this paper, this 
condition can
be substituted or relaxed.

\medskip
This paper is organized as follows. In Section \ref{section:maximum_p}, we prove a 
basic result,
which is a consequence of the Maximum Principle for the Laplacian, and which
is the key tool in the proof of our main results. The proof of Theorems
\ref{thm:main} and \ref{thm:main2}, assuming the existence of
a subsolution $h$, are presented in Section \ref{sect:proof_main}. In Section
\ref{sect:proof_lemma_1}, a proof of the existence of a subsolution (Lemma \ref{lemma:fundamental_lemma})
is given.
Finally, an example in the case
of a single equation where condition (a) is relaxed is presented in Section \ref{sect:last_remarks}.

\section{An important preliminary result}
\label{section:maximum_p}

The following theorem will be needed in our proof of Theorems \ref{thm:main}
and \ref{thm:main2}. It is rather well known, but we have included it for easy reference
and to give an explicit bound for the constant $c$ below.

\begin{theorem}[Extended Maximum Principle \cite{Cortissoz}]
\label{prop:maximump}
Let $\Omega$  be a smooth bounded open subset of $\mathbb{R}^n$ 
Let $h\in L^{\infty}\left(\Omega\right)$. Then the 
Dirichlet problem
\[
\left\{
\begin{array}{l}
-\Delta u = h \quad \mbox{in}\quad \Omega\\
u=0 \quad \mbox{on}\quad \partial \Omega,
\end{array}
\right.
\]
has a unique solution in $u\in W_0^{1,2}\left(\Omega\right)\cap L^{\infty}\left(\Omega\right)$.
Furthermore, we have
\[
\left\|u\right\|_{\infty}\leq c\left\|h\right\|_{\infty},
\]
where $c\leq \dfrac{d^2}{8}$, and $d$ is the slab diameter of $\Omega$.
\end{theorem}
\begin{proof}
(taken from \cite{Cortissoz})
The first part of the theorem is a direct consequence of Theorem 8.3 in \cite{GilbargTrudinger}.
The second part follows from Theorem 3.7 in \cite{GilbargTrudinger}
and an approximation argument. In fact, if $h$ is smooth,
assuming that $\Omega$ is contained in the slab
\[
\Pi_d=\left\{\left(x_1, \dots, x_{n-1}, x_n\right)\in \mathbb{R}^n\,:\, -\frac{d}{2}<x_{n}
<\frac{d}{2}\right\},
\]
and taking
\[
v= \frac{1}{2}\left(\frac{d^2}{4}-x_n^2\right)\left\|h\right\|_{\infty},
\]
in the proof of Theorem 3.7 in \cite{GilbargTrudinger}, we obtain the result for
$h$ smooth. Next, if $h\in L^{\infty}\left(\Omega\right)$, take a sequence of
smooth functions $h_n$ that converge to $h$ in $L^p$, $1\leq p<\infty$, and such that
$\left\|h_n\right\|_{\infty}\leq \left\|h\right\|_{\infty}$ 
(this can be done, for instance, using Friedrichs' mollifiers).
Let $u_n$ be the solution to the Dirichlet problem with $h_n$ on the right-hand side.
By the $L^p$ theory, $u_n\rightarrow u$ in $L^p\left(\Omega\right)$.
Since the $h_n$'s and $u_n$'s are smooth, we have that
\[
\left\|u_n\right\|_{\infty}\leq c\left\|h_n\right\|_{\infty},
\]
and hence
\[
\left\|u_n\right\|_{L^p\left(\Omega\right)}\leq 
c\mbox{Vol}\left(\Omega\right)^{\frac{1}{p}}\left\|h\right\|_{\infty},
\]
from we which can conclude that
\[
\left\|u\right\|_{L^p\left(\Omega\right)}\leq 
c\mbox{Vol}\left(\Omega\right)^{\frac{1}{p}}\left\|h\right\|_{\infty}.
\]
Since $p$ is arbitrary, we obtain the result.

\end{proof}

\section{Proof of the main results}
\label{sect:proof_main}

\subsection{Proof of Theorem \ref{thm:main}} 
Our proof of Theorem \ref{thm:main} is 
based on the following iteration:
\begin{equation*}
\left\{
    \begin{array}{l}
    \Delta u_{j,k+1}=\lambda_{j}\left(x\right)f\left(u_{0,k},u_{1,k},\dots, u_{n,k}\right), \,
    \mbox{in}\,\, \Omega, \quad j=0,1,2,\dots, n,\\
    u_{j,k+1}=0\quad \mbox{on} \quad \partial \Omega, \quad j=0,1,2,\dots,n.
    \end{array}
    \right.
\end{equation*}

Our first goal is to construct a set of functions $u_{j,0}\geq 0$, $j=0,1,\dots, n$,
such that 
\[
-\Delta u_{l,0}\leq \lambda_l\left(x\right)f_l\left(u_{0,0},u_{1,0},\dots, u_{n,0}\right),
\quad u_{l,0}=0 \quad \mbox{on}\quad \partial \Omega,
\]
to start the iteration (i.e., we must construct \emph{a subsolution} to system (\ref{eq:gen_lane_emden})). 
So we have the following:
\begin{lemma}
\label{lemma:fundamental_lemma}
Let $\lambda_l \in L^{\infty}\left(\Omega\right)$ be nonnegative functions and let
    $f_l: \mathbb{R}^n\longrightarrow \mathbb{R}$ (l=0,1,\dots, n) a family of 
    functions that satisfy (b) and (c) as stated in the introduction. 
    Assume also that there is a point $q\in \Omega$
    and $\rho>0$, $l=0,1,\dots,n$, such that 
    $\lambda_l\left(x\right)\geq \Lambda>0$ for all $x\in B_{2\rho}\left(q\right)$. 
    For any $\delta>0$ there exists a function 
    $h\in C^{2}\left(\Omega\right)\cap C\left(\overline{\Omega}\right)$
    such that $0<\left\|h\right\|_{L^{\infty}\left(\Omega\right)}<\delta$,
     $h=0$ on $\partial\Omega$, and
    \[
    -\Delta h \leq \lambda_l\left(x\right)f_{l}\left(h,\dots, h\right) \quad \mbox{in}
    \quad \Omega.
    \]
\end{lemma}
The proof of this lemma is rather long and technical and we shall give it later. 
Once we have Lemma \ref{lemma:fundamental_lemma}, the proof of Theorem \ref{thm:main}
can proceed as follows.

\medskip
We then set $u_{j,0}=h$ for all $j=0,1,2,\dots,n$, and given $u_{j,k}$, $j=0,1,\dots,n$,
define the $u_{j,k+1}$ as the solution to the Dirichlet problem
\begin{equation}
\label{eq:rec_gen_lane_emden}
\left\{
    \begin{array}{l}
    -\Delta u_{j,k+1}=\lambda_{j}\left(x\right)f_j\left(u_{0,k},u_{1,k},\dots, u_{n,k}\right), \quad j=0,1,2,\dots, n,\\
    u_{j,k+1}=0\quad \mbox{on} \quad \partial \Omega, \quad j=0,1,2,\dots,n.
    \end{array}
    \right.
\end{equation}
Using Theorem \ref{prop:maximump}, we know that each of these equations has a solution
in $W^{1,2}_0\left(\Omega\right)^{n+1}\cap L^{\infty}\left(\Omega\right)^{n+1}$.

\subsubsection{}
Now we show that the sequence $\left(u_{0,k},u_{1,k},\dots,u_{n, k}\right)$ converges towards a 
solution of (\ref{eq:gen_lane_emden}). We proceed as follows. First we show that
$u_{j,k}\leq u_{j,k+1}$, and then that, for each $j$, the sequence 
$\left\{u_{j,k}\right\}_{k=0,1,2\dots}$ is bounded in $L^{\infty}\left(\Omega\right)$-norm.

\medskip
The first claim follows from induction and the Classical Maximum Principle. Notice that by construction
\begin{eqnarray*}
-\Delta \left(u_{l,1}-u_{l,0}\right)&\geq &\lambda_{l}\left(x\right)f_l\left(u_{0,0},\dots,u_{n,0}\right)-
\lambda_{l}\left(x\right)f_l\left(h,\dots,h\right)\\
&=&\lambda_{l}\left(x\right)f_l\left(h,\dots,h\right)-\lambda_{l}\left(x\right)f_l\left(h,\dots,h\right)=0,
\end{eqnarray*}
and $u_{j,0}-u_{j,1}=0$ at $\partial \Omega$, and thus we must have by the Maximum Principle that $u_{j,1}\geq u_{j,0}$
in $\Omega$. Before going any further, we must
say a word about our use of the Maximum Principle: the inequalities involving the Laplacian
hold in the weak sense, and thus the deduced inequalities for the $u_{j,k}$'s hold almost everywhere
(see, for instance, Theorem 8.1 in \cite{GilbargTrudinger}).

\medskip
Once we have $u_{j,k}\geq u_{j,k-1}$, $j=0,1,\dots,n$, then we have
\[
-\Delta \left(u_{l,k+1}-u_{l,k}\right)\geq 
\lambda_{l}f_l\left(u_{0,k},\dots,u_{n,k}\right)
-\lambda_{l}f_l\left(u_{0, k-1},\dots,u_{n, k-1}\right)\geq 0
\]
and at the boundary $u_{j,k}-u_{j,k+1}=0$, and hence $u_{j,k+1}\geq u_{j,k}$ by the Maximum Principle. The induction
is complete, and the first part of the proof is complete as well.

\medskip
Next, we show that the iterates remain uniformly bounded. By Theorem \ref{prop:maximump},
\begin{eqnarray*}
\left\|u_{l,k+1}\right\|_{\infty}&\leq& c\Lambda \left\|f\left(u_{0,k},\dots,u_{n,k}\right)\right\|_{\infty} \\
&\leq & c\Lambda\sum_{j}B_{l,j} \left\|u_{j}\right\|_{\infty}^{p_{l,j}}+c\Lambda A.
\end{eqnarray*}
Therefore, for convenient constants $B'_{l,i}$ we have
\begin{eqnarray*}
1+\sum_l \left\|u_{l,k+1}\right\|_{\infty}&\leq&
2c\Lambda \sum_{l.i} B_{l,i}'\left(1+\sum_j \left\|u_{j,k}\right\|_{\infty}\right)^p\\
&\leq &
2\left(n+1\right)c\Lambda B\left(1+\sum_j \left\|u_{j,k}\right\|_{\infty}\right)^p,
\end{eqnarray*}
where $B=\max_{i,l}B_{i,l}'$, and $p=\max_{l,j}p_{l,j}$. Since $0<p<1$, if we pick
$h$ so that $\left\|h\right\|_{\infty}$ is small enough, this shows that the iterates $u_{j,k}$ are bounded
in $L^{\infty}\left(\Omega\right)$-norm.

\medskip
To finish our proof in the case of a bounded domain, we shall need the following lemma,
whose statement and proof (with the appropriate modifications),
and which we include here for the convenience of the reader, are taken from \cite{Cortissoz}. 
\begin{lemma}
\label{lemma:convergence}
Let $\{\left(u_{0,k},\dots,u_{n,k}\right)\}_{k}$ be a sequence in 
$L^{\infty}\left(\Omega\right)^{n+1}\cap W_0^{1,2}\left(\Omega\right)^{n+1}$, $\Omega$ a bounded domain, produced
from the iteration procedure described above, and such that
$u_{j,k} \uparrow u_j$ a.e.
Then $\left(u_0,u_1,\dots,u_n\right)$ is a weak solution to the BVP (\ref{eq:gen_lane_emden}). 
\end{lemma}
\begin{proof}
Assume that for all $j, k$, $\left|u_{j,k}\right|\leq D$.
Then, for any $m, l$ we have, 
$$\left|f\left(u_{0,k},\dots, u_{n,k}\right)\right|\leq E.$$ Also, from the system, multiplying both sides by $u_{j,m}-u_{j,l}$ and integrating by parts yields
\begin{eqnarray*}
\int_{\Omega}\left|\nabla\left(u_{j,m}-u_{j,l}\right)\right|^2\, dV
&=&\int_{\Omega} \lambda_{j} \left[f_j\left(u_{0,m-1},\dots, u_{n,m-1}\right)\right.\\
&&\left.-f_j\left(u_{0,l-1},\dots, u_{n,l-1}\right)\right]
\left(u_{j,m}-u_{j,l}\right)\,dV\\
&\leq& 2E\Lambda\int_{\Omega} \left|u_{j,m}-u_{j,l}\right|
\, dV,
\end{eqnarray*}
where $\Lambda=\max_{j}\left\|\lambda_{j}\right\|_{\infty}$.
As $u_{j,k}\rightarrow u_j$ a.e., by Egorov's theorem, given $\epsilon >0$, there is 
$\Omega'\subset \Omega$, such that 
\[
V\left(\Omega\setminus \Omega'\right)\leq \epsilon/\left(2E\Lambda+
2DV\left(\Omega\right)\right),
\]
where $V\left(\Omega\right)$ is the volume of $\Omega$, and 
such $\{u_{j,k}\}_k$ is uniformly Cauchy in $\Omega'$. Therefore, if $m,l$ are such that 
$$\left|u_{j,m}-u_{j,l}\right|\leq \epsilon / \left(2E\Lambda\right)$$
in $\Omega'$, then we have that
\[
\int_{\Omega}\left|\nabla\left(u_{j,m}-u_{j,l}\right)\right|^2\, dV
\leq 2\epsilon, 
\]
and thus the sequence is Cauchy in $W^{1,2}_0\left(\Omega\right)$, and in consequence it converges,
in $W^{1,2}_0\left(\Omega\right)$, say towards $u_j$. 
Let $\left(\varphi_0,\dots,\varphi_n\right)\in C_0^{\infty}\left(\Omega\right)^{n+1}$. Then, we have that
\[
\int_{\Omega}\nabla u_{j,m} \nabla \varphi_j\,dV=
\int_{\Omega}\lambda_{j}f\left(u_{0,m-1},\dots, u_{n,m-1}\right)\varphi_j\,dV,
\]
and hence, by the dominated convergence theorem, we
can take the limit $m\rightarrow\infty$ on both sides, and since
$u_{j,m-1}\uparrow u_j$, this shows that
$u_j$, $j=0,\dots,n$, is a weak solution to the BVP (\ref{eq:gen_lane_emden}).
\end{proof}

From the previous lemma it holds that the statement of Theorem \ref{thm:main} holds for bounded domains.
Also, and this is quite important, \emph{we can bound the solutions to (\ref{eq:gen_lane_emden}) in terms of an 
absolute constant that depends only on the assumptions on the $f_l$'s, on the maximum 
norm of the $\lambda_l$'s, and on the slab diameter of
the domain}.

\subsection{Regularity}
Next we prove that if $\lambda$ and $f$ are H\"older continuous, then the solution constructed above 
is classical:
this will be important to prove Theorem \ref{thm:main2}.
Since the $u_j$'s obtained are bounded,
each $f_l\left(u_1,\dots,u_n\right)\in L^{\infty}\left(\Omega\right)$. Hence,
by Theorem 8.33 in \cite{GilbargTrudinger}, any weak solution to 
\[
\Delta w = \lambda_l\left(x\right)f_l\left(u_1,\dots,u_n\right)
\]
is $C^1$. Thus, the functions $u_j$ are $C^1$, and hence 
each of the functions $f_l\left(u_1,\dots,u_n\right)$, 
as functions of $x\in \Omega$ are H\"older continuous, and again, by elliptic regularity,
this implies that the $u_j$'s are $C^2$.

\subsection{Proof of Theorem \ref{thm:main2}} 
In the case of a thin unbounded domain $\Omega$, we proceed as follows. 
Let $\left\{\Omega_k\right\}_{k=1,2,\dots}$
be a family of open subsets of $\Omega$ such that 
\begin{itemize}
\item each $\Omega_j$ is a bounded subset of $\Omega$, 

\item $\Omega_{k}\subset \Omega_{k+1}$, 

\item $\partial \Omega \cap \partial \Omega_j$ is an open subset of $\partial\Omega$, 

\item 
$\partial \Omega_k \cap \partial \Omega\subset\partial \Omega_{k+1} \cap \partial \Omega$,

\item 
$\bigcup \left(\partial \Omega_k \cap \partial \Omega\right)=\partial \Omega$
and $\bigcup_k \Omega_k=\Omega$.
\end{itemize}
A family such as the one described above can be 
constructed by intersecting $\Omega$ with large cubes and by ``smoothing the corners".
To get a solution to (\ref{eq:gen_lane_emden}) in $\Omega$,
solve the system (\ref{eq:gen_lane_emden}) in each of the $\Omega_k$'s defined above
(that is, apply Theorem \ref{thm:main}), and call 
the respective solutions $\left(u_{0,k},\dots,u_{n,k}\right)$. Notice that, being the slab diameter of
all the sets in the exhaustion being bounded above by the slab diameter of $\Omega$, there is
a $D>0$ such that 
\[
\left\|u_{j,k}\right\|_{L^{\infty}\left(\Omega_k\right)}\leq D.
\]
By the assumption on the regularity of the functions $\lambda_l$ and $f_l$, each 
of the solutions $u_{j,k}$ are $C^{2,\alpha}\left(\Omega_k\right)$ for certain $0<\alpha<1$.
Also, the Maximum Principle in this case implies that $u_{j,k}\leq u_{j,k+1}$ in the
intersection of their domains (that is $\Omega_k$). From these two facts, using the monotonicity property 
of the $f_l$'s (assumption (c)), the nonnegativity of the $\lambda_l$'s, and the H\"older continuity of the $f_l$'s,
it is not difficult to show that 
$\left(u_{0,k},\dots,u_{n,k}\right)$ converges pointwise to a vector-valued function
$\left(u_{0},\dots,u_{n}\right)$. 

\medskip
Next we show that $\left(u_{0},\dots,u_{n}\right)$ is a solution to (\ref{eq:gen_lane_emden}) in the domain $\Omega$.
By the local version of 
estimate (8.87) in Theorem 8.33 in \cite{GilbargTrudinger} 
(see Section 6.1 in \cite{GilbargTrudinger}), since $\left\|u_{j,k}\right\|_{L^{\infty}\left(\Omega_k\right)}\leq D$,
for $D$ independent of $j$ and $k$,
then for a fixed $k$, given a subdomain $L_k\subset \Omega_k$ with compact closure, $\left|u_{j,m}\right|_{C^{1,\alpha}\left(L_k\right)}\leq D'_k$, for $m\geq k$ with $D'_k$ independent of $m$ and from this follows that each component of
$\left(u_{0},\dots,u_{n}\right)$ belongs to $C_{loc}^{1,\beta}\left(\Omega_k\right)$. Hence, since the $f_l$'s
are locally H\"older continuous, by 
the Schauder estimates, the $u_{j,k}$'s belong to $C^{2,\beta}_{loc}\left(\Omega_k\right)$
and thus the limit functions
(the $u_{j}$'s) belong to $C_{loc}^{2,\frac{\beta}{2}}\left(\Omega\right)$ and in $\Omega$ they satisfy
\[
-\Delta u_l = \lambda_l\left(x\right)f_l\left(u_0,u_1,\dots,u_n\right) \quad \mbox{in} \quad \Omega,
\quad
l=0,1,\dots,n.
\]
Next we show that $\left(u_0,u_1,\dots,u_n\right)$ satisfies the boundary conditions. Let $p\in\partial\Omega$
and $\epsilon>0$
and take a ball of radius $r=\frac{\epsilon}{2D}>0$ centered at $p$ and consider the domain $L=\Omega\cap B_r\left(p\right)$. Then for some $K_0$,
$L\subset \Omega_k$ and $\partial L\cap \partial \Omega\subset \partial \Omega_k \cap \partial \Omega$
if $k\ge K_0$. Then, the previous considerations show that 
$\left|u_{j,k}\right|_{C^{1,\alpha}\left(\overline{L}\right)}\leq D$ for all $k\geq K_0$ (see Theorem 8.33 in 
\cite{GilbargTrudinger}). Thus, we can estimate por $p'\in L$
\begin{eqnarray*}
    \left|u_{j}\left(p'\right)\right|&=& \left|u_{j}\left(p'\right)-u_{j,k}\left(p\right)\right|\\
    &\leq& \left|u_{j}\left(p'\right)-u_{j,k}\left(p'\right)\right|+
    \left|u_{j,k}\left(p'\right)-u_{j,k}\left(p\right)\right|\\
    &\leq& \left|u_{j}\left(p'\right)-u_{j,k}\left(p'\right)\right|+Dr\\
    &\leq& \left|u_{j}\left(p'\right)-u_{j,k}\left(p'\right)\right|+\frac{\epsilon}{2}
\end{eqnarray*}
Thus, if we choose $k\geq K_0$ so that $\left|u_{j}\left(p'\right)-u_{j,k}\left(p'\right)\right|<\frac{\epsilon}{2}$,
then 
\[
 \left|u_{j}\left(p'\right)\right|<\epsilon.
\]
Since $p'\in L$ is arbitrary, we have shown that $u_{j}\left(x\right)\rightarrow 0$ as $x\rightarrow p$, which
is what we wanted to prove. This finishes the proof of Theorem \ref{thm:main2}
in the case of $\Omega$ being an unbounded thin domain.

\section{Proof of Lemma \ref{lemma:fundamental_lemma}}
\label{sect:proof_lemma_1}
Assume $0\in \Omega$.
Consider the standard bump function
Without loss of generality assume that $0\in \Omega$. 
Let $B_R\left(0\right)$ be the ball of radius $R>0$ centered at 0.
Choose
$R>0$ so that 
$\overline{B_R\left(0\right)}\subset \Omega$, all the $\lambda_l\left(x\right)\geq \Lambda>0$ in
$\overline{B_R\left(0\right)}$ (so we assume that $q=0$ and $\rho=R/2$), and
choose
$s \in (0,1)$ such that
\[
0<\alpha_{l,j}\leq s
\]
for all $l,j$.

Let $\phi\in C^\infty_c(\mathbb{R}^n)$ be defined by
\[
\phi(x) =
\begin{cases}
\exp\left( -\dfrac{1}{R^2 - |x|^2} \right), & \text{if } |x| < R, \\
0, & \text{if } |x| \geq R.
\end{cases}
\]
We compute the Laplacian of $\phi$:
\[
\Delta \phi(x) = \phi(x) \cdot
\left[ \left( \frac{2\left|x\right|}{(R^2 - \left|x\right|^2)^2} \right)^2 + \frac{2(R^2 + 3\left|x\right|^2)}{(R^2 - \left|x\right|^2)^3} + \frac{n-1}{\left|x\right|} \cdot \frac{2\left|x\right|}{(R^2 - \left|x\right|^2)^2} \right].
\]
Therefore
\[
\frac{-\Delta \phi(x)}{\phi(x)^s} = -\phi(x)^{1-s} \cdot \left[ \left( \frac{2\left|x\right|}{(R^2 - \left|x\right|^2)^2} \right)^2 + \frac{2(R^2 + 3\left|x\right|^2)}{(R^2 - \left|x\right|^2)^3} + \frac{n-1}{\left|x\right|} \cdot \frac{2\left|x\right|}{(R^2 - \left|x\right|^2)^2} \right].
\]
Since the expression between square brackets can be bounded above by a polynomial in
$\dfrac{1}{R^2-\left|x\right|^2}$,
$\phi$ vanishes rapidly (more so than any polynomial
in $R^2-\left|x\right|^2$) near the boundary and is smooth inside $B_R$, then
\[
-\Delta \phi\leq C \phi^s \quad \text{in } \Omega,
\]
for some constant $C > 0$ depending on $R$, $n$, and $s$.

\medskip
Let $0<\eta<1$ be such that $\eta\phi<\min\left\{\delta,1\right\}$ and $C\eta^{1-s}<\Lambda A$, where
\[
A =\min\left\{A_{l,j}:\, A_{l,j}\neq 0\right\}.
\]
Then we can estimate
\[
-\Delta (\eta\phi)\leq C\eta \phi^{s}
=C\eta^{1-s}\left(\eta \phi\right)^{s}
\leq \Lambda A\left(\eta \phi\right)^{s}.
\]
We claim that $h=\eta \phi$ is the subsolution we seek.
Indeed,
\begin{eqnarray*}
    -\Delta h &\leq& \Lambda A h^s \\
    &\leq& \Lambda\sum_{j=1}^n A_{l,j} h^s \quad\mbox{(by the definition of}\,\, A) \\
    &\leq & \Lambda \sum_{j=1}^n A_{l,j} h^{\alpha_{l,j}} \quad \text{(since}\,\, 0\leq h\leq 1)\\
    &\leq& \lambda_l\left(x\right)f_l\left(h,\dots,h\right),
\end{eqnarray*}
this last inequality justified by the hypothesis on $h$.

\section{Final Remarks}
\label{sect:last_remarks}

Condition (a) in the Introduction was used to prove that if we start
the iteration with $h$ small enough, the iterates are bounded
in the $L^{\infty}\left(\Omega\right)$-norm. However, condition (a) as it is, can be substituted 
by a different condition.

\medskip
Indeed, consider the Dirichlet problem,
\begin{equation}
\label{eq:Dirichlet_2}
-\Delta u = \lambda f\left(u\right) \quad \mbox{in}\quad \Omega,\quad
u=0 \quad \mbox{on} \quad \partial \Omega,
\end{equation}
with $\Omega$ a smooth
bounded domain, and $f\in L^{\infty}\left(\Omega\right)$.

\medskip
Assume that $f$ satisfies conditions (b)-(c) and the condition
\begin{itemize}
    \item[(a')] Assume that the function $\dfrac{d^2}{8}\lambda f\left(x\right)$ ($d$ is the
    slab diameter of $\Omega$) has a positive
    fixed point.
\end{itemize}
Then the Dirichlet problem (\ref{eq:Dirichlet_2}) has a positive solution. The proof goes as follows.
Let $x_0>0$ be a positive fixed point of $\dfrac{d^2}{8}\lambda f\left(x\right)$.
By (b), there is a function $h$ such that
\[
-\Delta h \leq \lambda f\left(h\right) \quad \mbox{in}\quad \Omega,\quad
h=0 \quad \mbox{on} \quad \partial \Omega.
\]
and $\left\|h\right\|_{\infty}<x_0$. Set $u_0=h$ and let $u_{n+1}$ be the 
solution to
\[
-\Delta u_{n+1}=\lambda f\left(u_n\right)
 \quad \mbox{in}\quad \Omega, \quad
u_{n+1}=0 \quad \mbox{on} \quad \partial \Omega.
\]
Then, by induction, it can be shown (as in Lemma 4 in \cite{Cortissoz2}) that
\[
\left\|u_{n}\right\|_{\infty}\leq x_0,
\]
and hence, arguing as before, the sequence $\left\{u_n\right\}_{n=0,1,2,\dots}$
converges towards a weak solution to (\ref{eq:Dirichlet_2}). If $f$ is locally H\"older,
this solution is a classical solution. Also, if $\Omega$ is a thin unbounded open set
($f$ locally H\"older),
just as before, the existence of a classical non-negative non-trivial solution can be shown. 

\medskip
A concrete example where the discussion above applies is given by the family of functions 
$f\left(x\right)=x^se^{x^m}$,
with $0<s<1$ and $m>0$. If $\lambda>0$ is small enough, then 
$\lambda x^s e^{x^m}$ has at least one positive fixed point. Finally, notice that
for fixed $\lambda$, if the domain is thin enough, then (a') holds, and thus
our arguments also apply: solvabilty is thus also related to the size of the domain.








\end{document}